\documentclass[]{siamart1116}
\usepackage[all]{xy}

\usepackage{tikz}

\usepackage{amssymb} 
\usetikzlibrary{automata, positioning, arrows}
\usepackage{color}
\usepackage{array}
\usepackage{longtable}

\numberwithin{theorem}{section}

\newcommand{\TheTitle}{A Tale of Two Limits: An Extremal Pagerank Problem} 
\newcommand{\TheAuthors}{J. Farnan and F. Kenter }

\headers{\TheTitle}{\TheAuthors}

\title{{\TheTitle}\thanks{Submitted to the editors April 8, 2021.
\funding{This work was funded in part by the NSF-DMS-1720225}}}

\title{A Tale of Two Limits, An Extremal Pagerank Problem}
\author{ and Franklin H. J. Kenter}
\date{\today} 

\author{
  Joseph Farnan 
\thanks{United States Naval Academy, Annapolis, MD,    \email{josephterencefarnan@gmail.com}}
  \and
  Franklin H. J. Kenter \thanks{United States Naval Academy, Annalpols, MD, \email{kenter@usna.edu}.}
}

\usepackage{amsopn}

\begin{document}

\maketitle

\def\P{{\mathbf{P}}}
\def\N{{\mathbb{N}}}
\def\Z{{\mathbb{Z}}}
\def\Q{{\mathbb{Q}}}
\def\K{{\mathbb{K}}}
\def\R{{\mathbb{R}}}
\def\C{{\mathbb{C}}}

\newcommand{\bfz}{{\bf z}}

\numberwithin{equation}{section}

\newcommand{\owari}{\hfill$\square$}

\theoremstyle{break}


 \newtheorem{problem}[theorem]{Problem}
 \newtheorem{fact}[theorem]{Fact}
 \newtheorem{conjecture}[theorem]{Conjecture}




\begin{abstract}
For a directed graph, the Pagerank algorithm emulates a random walker on the graph that occasionally ``jumps'' to a random vertex based on a jumping parameter $\alpha$. Upon completion, the algorithm generates a stochastic vector whose entries correspond to the limiting probability that the walker will be at that vertex. This vector is a right eigenvector of a corresponding Markov trasition matrix. Undoubtedly, this vector can drastically change based upon the jumping parameter $\alpha$.

In this article, we investigate the maximum possible discrepancy for different Pagerank vectors on the same unweighted directed (perhaps with loops) graph as measured by the 2-norm. We show that the limsup of this discrepancy can be as large as $\sqrt{\frac{67}{50}}$ using a very specific construction. (For contrast, the norm of the difference for any two stochastic vectors is at most $\sqrt{2}$.) Interestingly, on this construction this discrepancy occurs when $\alpha = 1$ and when $\alpha$ is very close to 1.
\end{abstract}



\section{Introduction} \label{Pagerank Introduction}

\subsection{Overview}

One way to rank the relative importance of vertices on a directed graph (network) is to use the Pagerank algorithm. The Pagerank algorithm was developed by Page and Brin to address the conundrum between quantity versus quality in the context of webpages \cite{pagerank}, as a webpage could be considered influential either because it is linked to by many sites or by a few, high-quality sites. For instance, many space-enthusiasts may link to a NASA, indicating its authority. In turn, NASA may link to a revered astrophysicist on its own, indicating its endorsement.

Pagerank and similar network algorithms have been developed for other contexts. In the of a sports league, the graph could have teams be the vertices and the directed edges represent a team losing against the other team. Previous work has utilized this idea of to rank NCAA College Football teams where it is often the case that some teams perform very well but with weaker opponents  \cite{callaghan2007random}. In social media, the graph would have users of vertices and the directed edges represent one user following or mentioning one another. In  \cite{kwak2010twitter}, the PageRank algorithm is used to provide an alternative measure of influence of Twitter users beyond simply counting the number of retweets. That is, who retweets our message is just as important as how many times it is retweeted.

Consider a directed graph, $G$. For our purposes and throughout, $G$ may have loops and bidirectional edges (i.e., there is an arc from $i$ to $j$ and another arc from $j$ to $i$). However, there will be no duplicate arcs (i.e., there is at most one arc from $i$ to $j$). The Pagerank algorithm produces a vector of length $n$ (the number of vertices in $G$), denoted $\pi$, where each entry measures on the proportion of the time we would encounter each vertex by simply wandering around the graph randomly and occasionally jump around. Hence, the Pagerank vector is a stochastic vector. 

To begin calculating the Pagerank vector, consider a directed graph $G$ and choose an $\alpha$ to be your {\it jumping parameter}, $0< \alpha <1$.
The parameter $\alpha$ will determine the proportion of the time we will follow an arc in our random process. Intuitively, $\alpha$ balances the impact of arcs from high-quality vertices and high quantities of arcs. 
Once $\alpha$ is chosen, randomly pick a vertex of $G$ uniformly at random to start, denoted $v$. Then, with probability $\alpha$, choose an out-neighbor of $v$, uniformly at random, and walk to that vertex; and with probability $1-\alpha$, jump to an another random vertex chosen uniformly at random. If $v$ has no out neighbors, then simply jump to a random vertex, chosen uniformly at random. Take the vertex we are on and declare that the new $v$ and repeat. The Pagerank vector is the limiting probability of finding ourselves at each vertex after an arbitrarily large number of steps. Intuitively, the more often we visit a vertex, the more important it is and the higher the (Page)rank it has!

For a given directed graph $G$ and $\alpha$, we define the Pagerank vector with jumping parameter $\alpha$
\[ (\pi_\alpha)_i = \lim_{t \to \infty} P(\text{being at vertex }i\text{ at step }t)\] for each vertex $i$.

The Pagerank can be interpreted using linear algebra in the context of Markov chains where the probability transition matrix for a directed graph $G$, $\mathbf{R}$, is given by
\[ \mathbf{R}_{ij} = \begin{cases}
\frac{\alpha}{{deg}^{out}_j} + \frac{1-\alpha}{n} ~ \text{ if }~ {deg}^{out}_j \ne 0 \\
\frac{1}{n} ~ \text{otherwise}
\end{cases}   \] 
Where $deg_{j}^{out}$ is the number of out neighbors of j.
In which case, the Pagerank vector $\pi$ is a right eigenvector of $\mathbf{R}$ with eigenvalue 1
\[ \mathbf{R} \pi = \pi \]
and is the stationary distribution of the random walk.

The existence and uniqueness of $\pi$ guaranteed by the Perron-Frobenius Theorem for all $0 \le \alpha < 1$ (see, for example, \cite{hornandjohnson}). For $\alpha = 1$,  $\pi_\alpha$ is well-defined provided $G$ 

\begin{itemize}
\item is aperiodic (i.e., the greatest common divisor of the length of all cycles is 1),
\item is weakly connected and 
\item there is a unique maximal induced subgraph that is strongly connected (perhaps with just a vertex with a single loop).
\end{itemize}

For simplicity, we will typically drop one of the subscripts of the Pagerank vector. We note that in this paper we will use subscripts to both the value of $\alpha$ and which vertex is being spoken of. 
For instance, we write $\pi_{\alpha = 1}$ to refer to the entire Pagerank vector when $alpha=1$ but also we write $\pi_A$ to denote to entry of the Pagerank vector corresponding to the vertex $A$; in the later case, the context of $\alpha$ will be clear.



\subsection{Problem Statement}

Even on the same graph, choosing different values for $\alpha$ can result in substantially different Pagerank vectors. For example, when $\alpha = 0$, the walker always jumps, in which case   
$\pi_\alpha$ is the uniform vector: 
$(1/n, 1/n, \ldots, 1/n)$ where $n$ is the number of vertices in $G$. On the other extreme, when $\alpha=1$, the walker always follows the arcs (unless they are stuck), so naturally, vertices with more in-arcs will be visited more often. In this situation, it is possible that $\pi_\alpha$ is concentrated at one vertex. This leads to the question: {\bf For the same graph, how different can these rankings be?}  Our main problem is to determine the maximum difference between Pagerank vector with one $\alpha$ and a Pagerank with a different $\alpha$ but on the same graph. Formally we ask:

\begin{problem}
How large can $$||\pi_{\alpha_{1}}-\pi_{\alpha_{2}}||_2$$
be over all directed graphs $G$ and $\alpha_1$, $\alpha_2$ with $0 \le \alpha_1, \alpha_2 \le 1$? Here, $\pi_{\alpha_{1}}$ is the Pagerank vector produced on the graph $G$ with jumping constant $\alpha_1$, and $\pi_{\alpha_{2}}$ is the ranking values returned by the Pagerank algorithm for the graph $G$ with jumping constant different $\alpha_2$.
\end{problem}

We focus on the 2-norm as it is the traditional norm, though we briefly discuss the problem for other norms in Section \ref{sec:open} which we leave as an open problem for the reader.

As we will discover, this problem is difficult because it requires carefully constructed examples with precisely chosen parameters $\alpha_1$ and $\alpha_2$. While we do not answer this problem explicitly, we provide a construction that provides an extremal value of $\sqrt\frac{67}{50}$.  Surprisingly, to achieve this, we use $\alpha_1=1$ and $\alpha_2$ very close to 1. We believe $\sqrt\frac{67}{50}$ to be best possible.

\subsection{Previous Similar Work}
Based on the linear-algebraic interpretation of Pagerank mentioned previously, this problem is fundamentally one about the principal eigenvector of graph-theoretic matrices. Indeed,  problems of a similar nature have been considered before. For an irreducible nonnegative matrix, the {\it principal ratio} is the ratio between the maximum and minimum entries of the principal eigenvector. For Markov transition matrices of simple directed graphs,  Askoy, Chung, and Peng show that the principal ratio of a random walk can be superexponentially small \cite{aksoy2016extreme}. For the adjacency matrix of undirected graphs, Tait proved that the ``kite graph'' (a complete graph with a long path) acheives the maximum principal ratio \cite{tait2017three}.

\subsection{Important facts and properties of PageRank}

To get our feet wet, here are some important properties of Pagerank vectors we use.

\begin{fact}
For $\alpha = 0$, $\pi_\alpha = (1/n, 1/n, \ldots, 1/n)'$. 
\end{fact}

Where $n$ is the number of vertices in the graph. What this says is that, in practice, one should choose an $\alpha$ away from 0 so that the resulting values vertices are more differentiated. Indeed, the original $\alpha$ for the Pagerank algorithm was $\alpha=0.85$ \cite{pagerank}.

In the context of maximizing $\| \pi_{\alpha=1} - \pi_{\alpha=x}\|_2$, this suggests that we would want to choose $\alpha$'s away from 0. Our the extremal choices for $\alpha$ are closer to 1. 

Also, in practice, people usually avoid choosing an $\alpha$ close to 1 because accurately computing $\pi_{\alpha \approx 1}$ numerically is challenging using methods such as the power method \cite{langville2011google}. Our graphical plots fail to show the true behavior near $\alpha=1$ because when $k$ is large, these numerical methods are not accurate enough! See for example figure \ref{fig:alphaplot}

\begin{fact} 
$||\pi||_2 \leq 1$.
\end{fact}
\begin{proof}
$||\pi||_2 \leq ||\pi||_1 = 1$.
\end{proof}
\begin{fact}
$||\pi_{\alpha}||_\infty \leq 1$ or in other words, $max(\pi_i) \leq 1$.  \hfill $\qed$
\end{fact}
\begin{fact} 
$||\pi_{\alpha_1} - \pi_{\alpha_2} ||_\infty \leq 1$.
\end{fact}
\begin{proof}
$||\pi_{\alpha_1} - \pi_{\alpha_2} ||_\infty \leq |max(|\pi_{\alpha_1} |, |\pi_{\alpha_2}|)|_\infty 
\leq 1$.
\end{proof}
\begin{fact}
$||\pi_1 - \pi_2||_2 \leq \sqrt{2}$
\end{fact}
\begin{proof}
$||\pi_1 - \pi_2||_2 \leq \sqrt{ ||\pi_1||_1 + ||\pi_2||_1} \leq \sqrt{2}$. The first inequality holds because the $\pi$s are non-negative. 
\end{proof}

This last bound is important because it provides an upper bound for how high $\| \pi_{\alpha=1} - \pi_{\alpha=x}\|_2$ can be. While we are unable to achieve $\sqrt{2} \approx 1.4142$, we are not too far from it: $\sqrt{\frac{67}{50}} \approx 1.1576$.

\section{Main Result}




\begin{theorem} \label{mainresult}
For any $\delta >0$, there exists an unweighted directed graph and $\alpha_1 = 1$ and $\alpha_2$ with 
\[ \| \pi_{\alpha_1}-\pi_{\alpha_2} \| \ge \sqrt{\frac{67}{50}} - \delta. \]

In particular, for the Pagerank vectors on the directed graph on $\Gamma(k)$ (defined in the next subsection),

\[ \lim_{k \to \infty} \| \pi_{\alpha_1=1}-\pi_{\alpha_2=1-1/k} \| = \sqrt{\frac{67}{50}}. \]
\end{theorem}

This answer is very interesting in several ways. First, the answer depends on the number of vertices in the graph in the sense that increasing the number of vertices in the graph enabled an slightly increased value of $||\pi_{\alpha_{1}}-\pi_{\alpha_{2}}||_2$. In this paper, we will show that $\displaystyle \lim_{k \to \infty}||\pi_{\alpha_{1}}-\pi_{\alpha_{2}}||_2$ (where $k + 3$ is number of vertices in the graph for very large $k$) goes to a certain size, but implicit in this is that increasing the number of vertices in the graph will allow a value closer to the limit. Also, the extremal difference occurs approximately when $\alpha = 1$ and $1 > \alpha \geq 1 - 1/k$. Figure \ref{fig:alphaplot} plots $\| \pi_{\alpha=1} - \pi_{\alpha=x}\|_2$ for a particular construction we give in the next section. The graph is beyond deceiving! When $x=1$, the value for $\| \pi_{\alpha=1} - \pi_{\alpha=x}\|_2 = 0$, but a large enough graph, the maximum value is achieved very very close to $x=1$!

We believe the value $\sqrt{\frac{67}{50}}$ is the best possible:

\begin{conjecture}

For \emph{any} unweighted directed graph (perhaps with loops and perhaps bidirected arcs) and two jumping constants $\alpha_1$ and $\alpha_2$, we have, 
\[ \| \pi_{\alpha_1}-\pi_{\alpha_2} \| < \sqrt{\frac{67}{50}}. \]
\end{conjecture}

We will give our reasoning and intuition for this conjecture in Section \ref{sec:open}

\subsection{Our Construction\texorpdfstring{, $\Gamma(k)$}{}}

Our construction is a long ``ladder'' with $k+3$ vertices: $C_1, C_2, B_1, \ldots, B_k,$ and $A$. It is illustrated in graph Figure $\ref{fig:6750}$ and is constructed as follows. Vertex $C_1$ has an arc to itself and to $C_2$. $C_2$ has an arc back to $C_1$, a loop to itself and a path to $B_1$. 
$B_1$ has an arc back to $C_1$ and $C_2$, and an arc to $B_2$. Likewise, for all $i < k$, $B_i$, has arcs back to both $C_1$ and $C_2$ and an arc to $B_{i+1}$. Then, $B_{k}$ which has an arc back to $C_1$ and $C_2$ and to a vertex $A$. However, vertex $A$ only has a loop to itself. 

This construction will end up concentrating the random walkers either in $A$ or in $C_1$ and $C_2$, depending on the value of the jumping constant $\alpha$. 

When $\alpha=1$, the random walker will always follow the arcs. The walker will climb up and fall down ``the ladder'', but inevitably, they will eventually make it to $A$ and become stuck. Therefore, when $\alpha =1$ we have $\pi_A = 1$. 

On the other hand, for other, carefully chosen, values of $\alpha$ close to 1, the random walker will frequently ``fall down'' to $C_1$ and $C_2$. In which case, it is next to impossible for the walker to get to $A$. And even if the walker gets to $A$, since $\alpha < 1$, the walker is able to avoid becoming stuck. Within the proof of Theorem \ref{mainresult}, we show that $\pi_A \to 0$ as $k \to \infty$ for $\alpha = 1 - 1/k$.




\begin{figure}
  \centering
\begin{tikzpicture}
    \node[state] (a) at (10,4) {$A$};
    \node[state] (b) at (8.5,4) {$B_k$};
    \node[state] (c) at (7,4) {$\cdots$};
    \node[state] (d) at (5.5,4) {$B_1$};
    \node[state] (e) at (4,4) {$C_2$};
    \node[state] (f) at (2.5,4) {$C_1$};

    \path[<-] (a) edge (a);
    \path[->] (b) edge (a);
    \draw[<-] (b) edge (c);
    \draw[<-] (c) edge (d);
    \draw  (e) edge [<-, out=30, in=30] (b);
    \draw  (f) edge [<-, out=30, in=30] (c);
    \draw  (e) edge [<-, out=30, in=30] (c);
    \draw[->] (e) edge (f);
    \draw (f) edge [<-, out=30, in=30]  (b);
    \draw (f) edge [<-, in=30, out=30 ]  (d);
    \draw[<-] (d) edge (e);
    
    \draw[<-] (e) edge (f);
    \path (a) edge [loop below] (a);
    \path (e) edge [loop below] (e);
    \path (f) edge [loop left] (f);
    \draw[->] (d) edge  (e);
\end{tikzpicture}
\newline
  
    \caption{Construction of the family of graphs used in the proof of Theorem \ref{mainresult}}
    \label{fig:6750}
\end{figure}
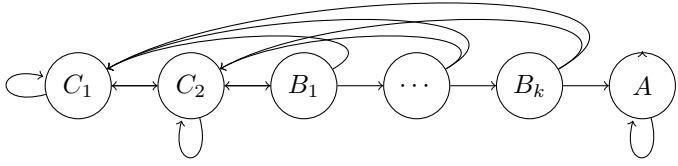

\subsection{Proof of Theorem \ref{mainresult}}


Our proof is outlined as follows. We will start by explicitly computing $\pi_{\alpha_1=1}$ for $\Gamma(k)$ for any $k > 0$. We then compute  $\pi_{\alpha_2=1-1/k}$ and show that $\pi_{C_1}$ and $\pi_{C_2} \to \frac{2}{5}$ as $k \to \infty$. We then analyze $\pi_{B_i}$ and show that $\sum_i \pi_{B_i}^2$ can be approximated by a geometric series in terms of $\pi_{C_2}$. As a result, we will have that as $k\to \infty$, $\|\pi_{\alpha_1=1}-\pi_{\alpha_2=1-1/k}\|_2 ^2 \to 67/50$.
~\\



\subsubsection{Computing \texorpdfstring{$\pi_{\alpha=1}$}{when alpha=1}} 

\begin{lemma} 
When $\alpha=1$, $\pi_A = 1$, and $\pi_v = 0$ for any vertex $v \ne A$.
\end{lemma}
\begin{proof}
The loop at $A$ ensures that when $\alpha = 1$, a ``random-walker'' must follow their way to $A$ and remain there. Therefore, $\pi_{A}$ has a value of 1 at vertex $A$ and 0 at all other vertices. 
\end{proof}

Now, the more difficult task, of carefully computing $\pi_{\alpha_2 = 1-1/k}$ remains.

\subsubsection{Computing \texorpdfstring{$\pi_A$ when $\alpha=1-1/k$}{the stationary distribution of A}}

We will show that as $k\to\infty$, 

\[ \pi_{C_1} \to \frac{2}{5}, \quad \pi_{C_2} \to \frac{2}{5}, \text{ and } \quad \pi_{A}  \to 0. 
\]



 
 Our main technique for computing $\pi_v$ is to consider change with each time step for sufficiently large $t$. We let $\Delta \pi_v = \sum_u \left( \mathbf{R}_{v,u} \pi_u \right) - \pi{v}$. Effectively, $\Delta \pi_v$ is ``rate in'' - ``rate out''.    
 
\begin{lemma}
For any vertex $v$, the stationary distribution obeys $\Delta \pi_v = 0$. 
\end{lemma}
\hfill $\square$ 
~\\
We can apply the previous lemma to compute $\pi_A$.

For a random walker to reach $A$ there are 4 disjoint possibilities:

\begin{itemize}
    \item $k$ steps ago, the random walker was at $C_2$ and did not jump within the intervening $k$ steps from $C_2$ to A. 
    \item 1 step ago, the walker was at $A$ and then they took the loop back to $A$.
    \item $k-i$ steps ago, the walker jumped to $B_i$ and the made it up the ``ladder'' to $A$ without jumping again, or
    \item the walker could have directly jumped to $A$.
\end{itemize}

Therefore, we have

\[ \Delta \pi_A = \left(\left(1-\frac1k\right)^k  \pi_{C_{2}} (1/3)^k + \left(1-\frac1k\right) \pi_A  + \sum_{i=1}^{k} \frac{\frac1k (1/3)^{k-i} \left(1-\frac1k \right)^{k-i}}{3+k} \right) - \pi_A.  \]

The first term accounts for the first possibility. The second term accounts for the second possibility. The sum accounts for the last two possibilities. Finally, $-\pi_A$ falls from the definition of $\Delta \pi_A$ and contains all of the walkers leaving $A$ even if they come back.

This simplifies to 

\[ \Delta \pi_A = \left(\left(1-\frac1k\right)^k  \pi_{C_{2}} (1/3)^k + -\frac1k \pi_A  + \sum_{i=1}^{k} \frac{\frac1k (1/3)^{k-i} (1-\frac1k )^{k-i}}{3+k} \right).  \]

~\\

{\noindent \it Claim:} \[ \pi_A \to 0 \text{ as } k \to \infty \]

\begin{proof}
For the stationary distribution,
\[ \Delta \pi_A = (1-\frac1k)^k  \pi_{C_2} (1/3)^k - \pi_A /k  + \sum_{i=1}^{k} \frac{\frac1k (1/3)^{k-i} (1-\frac1k)^{k}}{3+k}.  \]
Observe that the series
\[\sum_{i=1}^{k}\left (\frac{1/k~(1/3)^{k-i} (1-1/k)^k}{3+k}\right).\]
is a finite geometric series that's bounded above by its infinite sum which is \[ \frac{3/k}{2(3+k)}.\] 
 
And solving for $\pi_A$ we have \[ \pi_A < \frac{(1-1/k)^k \pi_{C_2} (1/3)^k}{1/k} + \frac{3}{2(3+k)}. \]
 Since $0 \le \pi_{C_2} \le 1$ and $1/k > 0$, we can simplify the above equation


\begin{equation} \label{pi_a2}
     \pi_A < \frac{(1-1/k)^k (1/3)^k}{1/k} + \frac{3}{2(3+k)}< \frac{ (1/3)^k}{1/k} + \frac{3}{2(3+k)} = \frac{k}{3^k} + \frac{3}{2(k+1)}. 
\end{equation}


 
Therefore, as $k \to \infty$, $\pi_A \to 0$

\end{proof}





\subsubsection{Relating \texorpdfstring{$\pi_{C_1}$, $\pi_{C_2}$ and $\sum_i \pi_{B_i}$ when $\alpha=1-1/k$}{the stationary distribution various vertices}} 

{\it Claim:}
\[ \pi_{C_1} = \pi_{C_2} \]
and
\[  \pi_{C_1} = \frac{1/k}{3+k} + (1-1/k) \cdot \frac{\sum_{i=1}^{k} \pi_{B_i}}{3}+ \pi_{C_1} /3 + (1-1/k) \cdot \pi_{C_1} /2 \]
and hence 
\[  \pi_{C_2} = \frac{1/k}{3+k} + (1-1/k) \cdot \frac{\sum_{i=1}^{k} \pi_{B_i}}{3}+ \pi_{C_2} /3 + (1-1/k)\cdot\pi_{C_2} /2 \]

\begin{proof}
Now, consider $\Delta \pi_{C_1}$ and $\Delta \pi_{C_2}.$
For $C_1$, the vertex at the bottom of the chain of vertices, we have

\[ \Delta \pi_{C_1} = \frac{1/k}{3+k} + (1-1/k)\cdot \frac{\sum_{i=1}^{k} \pi_{B_i}}{3}+ \pi_{C_2} /3 + (1-1/k)\cdot\pi_{C_1} /2 - \pi_{C_1}\]



This follows similarly to the case for $A$ and $\pi_A$ previously. The first term is represents mass jumping from some vertex to $C_1$; the second term is the probability mass that is at $B_i$ and returns to $C_1$; the third term is the mass from $C_2$ that goes back to $C_1$; the fourth term is the probability that $C_1$ gets from its self loop; and the fifth term is everything leaving $C_1$. 


For $\Delta\pi_{C_1} = 0$, we have 

\[  \pi_{C_1} = \frac{1/k}{3+k} + (1-1/k) \cdot \frac{\sum_{i=1}^{k} \pi_{B_i}}{3}+ \pi_{C_2} /3 + (1-1/k)\cdot\pi_{C_1} /2 \]

Similarly for $\pi_{C_2}$, we have
\[ \Delta \pi_{C_2} = \frac{1/k}{3+k} + (1-1/k) \cdot \frac{\sum_{i=1}^{k} \pi_{B_i}}{3}+ (1-1/k)\cdot \pi_{C_2} /3  + (1-1/k)\cdot \pi_{C_1} /2 - \pi_{C_2} \]
\text{ and } 
\begin{equation}  \pi_{C_2} = \frac{1/k}{3+k} + (1-1/k) \cdot \frac{\sum_{i=1}^{k} \pi_{B_i}}{3}+ (1-1/k)\cdot \pi_{C_2} /3 + (1-1/k)\cdot\pi_{C_1} /2 \label{C2}
\end{equation}

The symmetry for the equations for $\pi_{C_1}$ and $\pi_{C_2}$ give that
\[ \pi_{C_1} = \pi_{C_2} \] 
\end{proof}


Now, let us relate $\pi_{C_2}$ to the $\pi_{B_i}$.

We can substitute 
$\pi_{C_2} $ in for $\pi_{C_1}$ in 
the equation for $\pi_{C_2}$
\[  \pi_{C_2} = \frac{1/k}{3+k} + (1-1/k)\cdot \frac{\sum_{i=1}^{k} \pi_{B_i}}{3}+ (1-1/k)\cdot \pi_{C_2} /3 + (1-1/k)\cdot(\pi_{C_2}) /2 \]
\begin{equation}   \pi_{C_2} = \frac{1/k}{3+k} + (1-1/k)\cdot \frac{\sum_{i=1}^{k} \pi_{B_i}}{3}+ \frac{(1-1/k)\cdot \pi_{C_2} \cdot5}{6} \label{C2v} \end{equation}





Finally, let us each consider the $\pi_{B_i}$.
\[ \Delta \pi_{B_i} = \frac{1/k}{3+k} + (1-1/k)\cdot \frac{\pi_{B_{i-1}}}{3} - \pi_{B_i}\]
And thus for the stationary distribution,

\[  \pi_{B_i} = \frac{1/k}{3+k}+ (1-1/k)\cdot \frac{\pi_{B_{i-1}}}{3} \]

And also for $B_1$,
\[  \pi_{B_1} = \frac{1/k}{3+k}+ (1-1/k)\cdot \frac{\pi_{C_2}}{3} \]

\[  \pi_{B_i} > (1-1/k) \frac{\pi_{B_{i-1}}}{3} \]
and
\[  \pi_{B_1} > (1-1/k) \frac{\pi_{C_2}}{3}. \]
So by induction, 
\begin{equation}\label{Bi}
\pi_{B_i} > (1-1/k)^i \frac{\pi_{C_2}}{3^i} 
\end{equation}
and 

\[ \sum_{i=1}^k \pi_{B_i} > \sum_{i=1}^k (1-1/k)^i \frac{\pi_{C_2}}{3^i} = \frac{(\pi_{C_2}-\pi_{C_2}(\frac{1-1/k}{3})^k)}{(1-\frac{1-1/k}{3})}\]
This equation comes from viewing the previous equations as a finite geometric series.

\subsubsection{Showing \texorpdfstring{$\pi_{C_1}, \pi_{C_2} \to \frac{2}{5}$}{C1 and C2 converge to 2/5}}

We will now evaluate the limit with our choice $\alpha = 1 - \frac{1}{k}$. We already know that as $k \to \infty$, $\pi_A \to 0$; now we show that $\pi_{C_1} \to 2/5$, $\pi_{C_2} \to 2/5$.


 



{\it Claim:} 
\[  \pi_{C_1} = \pi_{C_2} \to \frac{2}{5} \]

\begin{proof}

We will show that $ \lim_{k \to \infty} \pi_{C_2} \le \frac25$ and that $ \lim_{k \to \infty}\pi_{C_2} \ge \frac25$

~\\


Let's start with
\[ \pi_{C_1} + \pi_{C_2} + \pi_A + \sum_{i=1}^{k} \pi_{B_i} =1.\]
and so
\[ \lim_{k \to \infty} \left[ \pi_{C_1} + \pi_{C_2} + \pi_A + \sum_{i=1}^{k} \pi_{B_i}\right] =1.\]


By (\ref{C2}) we can substitute
\begin{eqnarray*} 
\pi_{C_1} &=& \pi_{C_2}
\end{eqnarray*}
into the above:
\[ \lim_{k \to \infty} \left[ 2 \pi_{C_2} + \pi_A + \sum_{i=1}^{k} \pi_{B_i}\right] =1.\]
and by distributing the limits:

\[  2 \left[ \lim_{k \to \infty}  \pi_{C_2} \right] + \left[ \lim_{k \to \infty} \pi_A \right] + \left[ \lim_{k \to \infty} \sum_{i=1}^{k} \pi_{B_i}\right] =1.\]







Since $\pi_A \to 0$, We can find the total ranking as 
\[ \sum_{i=1}^k \pi_{B_i} + \pi_{C_1} + \pi_{C_2} \to 1.\]

To find an upper bound for $\pi_{C_2}$ we can sum the Pagerank values in $C_2$, $C_1$ and the value of the $\pi_{B_i}$s that comes from $C_2$. This gives us:
\[ 1> \pi_{C_1} + \frac{((\pi_{C_2}-\pi_{C_2}\cdot(\frac{1-1/k}{3})^k)}{(1-\frac{1-1/k}{3})} \]

By taking the limit of both sides

\[ \ \lim_{k \to \infty}  1 \ge \lim_{k \to \infty} \left[   \pi_{C_1} + \frac{(\pi_{C_2}-\pi_{C_2} \cdot (\frac{1-1/k}{3})^k}{(1-\frac{1-1/k}{3})}\right] \]

\[  1 \ge \lim_{k \to \infty} \left[ \pi_{C_1} \right] + \lim_{k \to \infty} \left[ \frac{(\pi_{C_2}-\pi_{C_2} \cdot (\frac{1-1/k}{3})^k}{(1-\frac{1-1/k}{3})}\right] \]


Splitting the fraction and substituting $\pi_{C_2}$ in for $\pi_{C_1}$ 
\[ 1> \lim_{k \to \infty} (\pi_{C_2}) + \left[ \lim_{k \to \infty} \frac{\pi_{C_2}} {(1-\frac{1-1/k}{3})}\right] - \lim_{k \to \infty} \left[ \frac{\pi_{C_2} \cdot(\frac{1-1/k}{3})^k}{(1-\frac{1-1/k}{3})}\right] \]

Taking the limits
\[ 1 \geq \lim_{k \to \infty} \pi_{C_2} + \lim_{k \to \infty} \frac{\pi_{C_2}} {\frac{2}{3}}- 0 \]

And so by solving for $\lim_{k \to \infty} \pi_{C_2}$,
\[ 2/5 \geq \lim_{k \to \infty} \pi_{C_2} \]

We now find a lower bound for $\pi_{C_2}$.

Let's start with (\ref{C2}):
\[ 
\pi_{C_2} = \frac{1/k}{3+k} + (1-1/k)\cdot \frac{\sum_{i=1}^{k} \pi_{B_i}}{3}+ \frac{(1-1/k)\cdot \pi_{C_2} \cdot5}{6}
\]
By removing the first term, 
\[\pi_{C_2} >  (1-1/k) \frac{\sum_{i=1}^k \pi_{B_i}} {3}+ \frac{5 (1-1/k) \pi_{C_2} }{6} ) \]
Combining the $\pi_{C_2}$ terms,
\[\pi_{C_2}- \frac{ \pi_{C_2} \cdot5}{6} + \frac{5\pi_{C_2}}{6k} >   (1-1/k)\cdot\frac{\sum_{i=1}^k \pi_{B_i}} {3} \]
Simplifying
\[ \frac{\pi_{C_2}}{6}+\frac{5\pi_{C_2}}{6k} >   (1-1/k)\cdot\frac{\sum_{i=1}^k \pi_{B_i}} {3} \]
Therefore,
\begin{equation}\label{upperbi}
\sum_{i=1}^k \pi_{B_i} < \left( \frac{\pi_{C_2}}{6} +\frac{5\pi_{C_2}}{6k}\right)\cdot\frac{3}{(1-1/k)} 
\end{equation}
By applying this, 
\[  \pi_{C_1} + \pi_{C_2} + \pi_A + \sum_{i=1}^k \pi_{B_i} =1\]

\begin{eqnarray*}1 &<& \pi_{C_2} + \pi_{C_2} + \frac{k}{3^k} + \frac{1}{2k}  \\&& +
\left( \frac{\pi_{C_2}}{6}+\frac{5\pi_{C_2}}{6k} \right) \frac{3}{(1-1/k)}
\end{eqnarray*}

Which will help us to find $\pi_{C_2}$ since
\begin{eqnarray*}1 &<& \pi_{C_2}  + \pi_{C_2} + \frac{\pi_{C_2}}{2(1-1/k)} + \frac{k}{3^k} + \frac{1}{2k} + \frac{5 \pi_{C_2}}{2k-2} 
\end{eqnarray*}
Or solving for $\pi_{C_2}$, 
\[ \pi_{C_2} > \frac{2}{5} - \frac{\pi_{C_2}}{2}(\frac{k}{k-1} - 1) -\frac{k}{3^k} - \frac{1}{2k} - \frac{5 \pi_{C_2}}{2k-2}
\]
Let $d =\frac{\pi_{C_2}}{2}(\frac{k}{k-1} - 1) + \frac{k}{3^k} + \frac{1}{2k} + \frac{5 \pi_{C_2}}{2k-2}$. We do this because as $k \to \infty, d \to 0$ and because $\pi_{C_2} > \frac{2}{5} - d$ which implies:


\[ \lim_{k \to \infty} \pi_{C_1} = \lim_{k \to \infty} \pi_{C_2} = \frac{2}{5}  \]
\end{proof}

\subsubsection{Computing \texorpdfstring{$\| \pi_{\alpha=1-1/k} \|^2$}{the 2-norm of one Pagerank vector}} 

To compute $\| \pi_{\alpha=1-1/k} \|^2$, we are not necessarily concerned with what $\sum_{i=1}^k \pi_{B_i}$ is equal to, rather we are more concerned about what $\sum_{i=1}^k \pi_{B_i} ^2$ is equal to.

Using the fact that $\pi_{C_1} = \pi_{C_2}$ and $\pi_{C_2} \to 2/5$ and $\pi_{C_2} + \sum_{i=1}^k \pi_{B_i}$ goes to a geometric series we now have the following,

\begin{lemma}
\[ \lim_{k \to \infty} \|\pi_{\alpha=1-1/k} \|_2^2 \to \frac{17}{50} \]
\end{lemma}
\begin{proof}
As before, we will provide an lower bound and then an upper bound for  $\lim_{k \to \infty} \|\pi_{\alpha=1-1/k} \|_2^2.$

Let us begin with the lower bound. 


\begin{eqnarray*}
\|\pi_{\alpha=1-1/k}\|_2^2 &=& \pi_A ^2 + \pi_{C_1}^2 + \pi_{C_2}^2 + \sum_i^k \pi_{B_i}^2  \\
&>& \pi_{A} ^2
+\pi_{C_1} ^2 + \pi_{C_2} ^2 +  \sum_{i=1}^k  (1-1/k)^i \frac{\pi_{C_2}}{3^i} \\
&=& \pi_{A} ^2
+\pi_{C_1} ^2 +\frac{\pi_{C_2} ^2 - \pi_{C_2}^2 (\frac{1-1/k}{3})^{2k}} {1-\left(\frac{1-1/k}{3}\right)^2}
\end{eqnarray*}

Where the first inequality follows from (\ref{Bi}). The second equality applies the finite geometric series.

So by applying the limit on both sides, we have 

\begin{eqnarray*} \lim_{k\to \infty} \|\pi_{\alpha=1-1/k}\|_2^2  &\ge& \lim_{k \to \infty} \left( \pi_{A} ^2
+\pi_{C_1} ^2 +\frac{\pi_{C_2} ^2 - \pi_{C_2}^2 (\frac{1-1/k}{3})^{2k}} {1-\left(\frac{1-1/k}{3}\right)^2} \right) \\ 
&=& \lim_{k \to \infty}  \pi_{A} ^2 
+ \lim_{k \to \infty}  \pi_{C_1} ^2 + \lim_{k \to \infty} \frac{\pi_{C_2} ^2 - \pi_{C_2}^2 (\frac{1-1/k}{3})^{2k}} {1-\left(\frac{1-1/k}{3}\right)^2} \\ 
&=& 0 + \frac{4}{25} + \frac{4 \cdot 9}{25 \cdot 8} 
\end{eqnarray*}
To find an upper bound, we only need an upper bound for the $\pi_{B_i}$. We can observe that
\[
\pi_{B_i} = (1/3)^i \pi_{C_2} \left(1-\frac1k\right)^i + \sum_{n = 1}^{i-1} \frac{(1/3)^{i-n} (1- \frac{1}{k})^{i-n} \cdot \frac1k)}{3+k}.
\]
The first term above contains all of the probability mass that has gone through $C_2$ since it last jumped and the second term contains the probability mass that has not (i.e., it jumped onto vertex $B_j$ for $j<i$).

By overestimating the second term of the equation above with an infinite geometric series, it is bounded by $\frac{3/k}{2(3+k)}$. Ergo
\[
\pi_{B_i} < (1/3)^i \pi_{C_2} \left(1-\frac1k\right)^i + \frac{3/k}{2(3+k)}.
\]
However, we are concerned about $\pi_{B_i} ^2$ so by expanding we have
\[
\pi_{B_i} ^2 < (1/3)^{2i} \pi_{C_2}^2 (1-1/k)^{2i}  + (1/3)^i \pi_{C_2} (1-1/k)^i \frac{3/k}{2(3+k)} + \left (\frac{3/k}{2(3+k)}\right )^2.
\]
Since $(1/3)^i \pi_{C_2} (1-1/k)^i < 1$ we know that 
\[
\pi_{B_i} ^2 <  (1/3)^{2i} \pi_{C_2}^2 (1-1/k)^{2i}  + \frac{3/k}{2(3+k)} + (\frac{3/k}{2(3+k)})^2.
\]

So
\begin{eqnarray*}
&~& \|\pi_{\alpha=1-1/k}\|_2^2 \\ &=& \pi_{C_1}^2 + \left( \pi_{C_2}^2 + \sum_{i=1}^k \pi_{B_i}^2 \right) + \pi_A^2. \\
&<& \pi_{C_1}^2 + \left( \pi_{C_2}^2 + \sum_{i=1}^k \left[(1/3)^{2i} \pi_{C_2}^2 (1-1/k)^{2i}  + \frac{3/k}{2(3+k)} + (\frac{3/k}{2(3+k)})^2 \right] \right) + \pi_A^2. \\
&=& \pi_{C_1}^2 + \left( \pi_{C_2}^2 + \sum_{i=1}^k (1/3)^{2i} \pi_{C_2}^2 (1-1/k)^{2i} \right) + k \frac{3/k}{2(3+k)} + k (\frac{3/k}{2(3+k)})^2  + \pi_A^2. \\
&<& \pi_{C_1} ^2 + \frac{\pi_{C_2}^2}{1-(1/9)} + k\left (\frac{3/k}{2(3+k)} + \left(\frac{3/k}{2(3+k)}\right)^2\right) + \pi_A^2. \\
\end{eqnarray*}
where the first inequality follows using the previous estimate of $\pi_{B_i}$, the second inequality follows from overestimating the geometric sum as an infinite series.


So by applying the limit, we have

\[  \lim_{k \to \infty} \|\pi_{\alpha=1-1/k}\|_2^2 \le \lim_{k \to \infty} \left[\frac{4}{25} + \frac{4 \cdot 9}{25 \cdot 8} + \frac{3}{2(3+k)} + \frac{9/k}{4(3+k)^2}  + \pi_A^2 \right ]
\]
which goes to 
\[ \frac{4}{25} + \frac{4 \cdot 9}{25 \cdot 8}  = \frac{17}{50}
\]
as $k \to \infty$.
\end{proof}

\subsubsection{Computing \texorpdfstring{$\|\pi_{\alpha=1}-\pi_{\alpha=1-1/k}\|_2$}{the norm of the difference}}

Since when $\alpha=1$ vertex $\pi_A = 1$, and $\alpha = 1 - 1/k$, $\pi_A \to 0$, the sum
\[ \|\pi_{\alpha=1}-\pi_{\alpha=1-1/k}\|_2 = \sum_{i=1}( \pi_{\alpha=1, v} - \pi_{\alpha=1-1/k, v} )^2
\]
has either $\pi_{\alpha=1, v} \to 0 $ or $\pi_{\alpha=1-1/k, v} \to 0$, so \[  \lim_{k \to \infty} \|  \pi_{\alpha=1-1/k}-\pi_{\alpha=1} \|^2 =  \lim_{k \to \infty} \| \pi_{\alpha=1-1/k}\|^2 +\lim_{k \to \infty} \| \pi_{\alpha=1} \|^2 . \]
We have
\[  \lim_{k \to \infty} \| \pi_{\alpha=1-1/k}-\pi_{\alpha=1} \| \to \sqrt{\|\pi_{\alpha=1-1/k}\|^2 + \|\pi_{\alpha=1}\|^2} \to  \sqrt{\frac{17}{50} + 1} = \sqrt{\frac{67}{50}} \] 
\hfill $\triangle$

Indeed the deceiving graph from Figure \ref{fig:alphaplot} shows $\| \pi_{\alpha=1-1/k}-\pi_{\alpha=1} \| \to \sqrt{\frac{67}{50}} \approx 1.1575$ as $\alpha \to 1$. 
\begin{figure}[ht]
    \centering
    \includegraphics[width=4in]{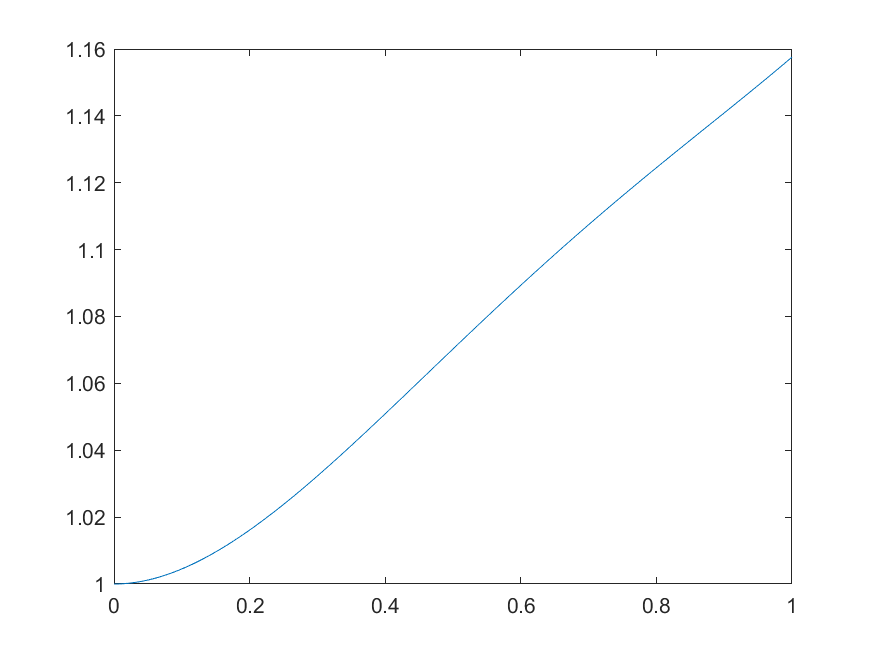}

    \label{fig:alphaplot}
    \caption{A plot of $x =  \alpha$ versus $y=||\pi_{\alpha = {1}}-\pi_{\alpha = x}||$ using the construction from Theorem \ref{mainresult}, $\Gamma(k)$, for sufficiently large $k$. The resolution of the plot is unable to capture that $|\pi_{\alpha = {1}}-\pi_{\alpha = x}|| \to 0$ as $x \to 1$. }
\end{figure}

\section{The Open Problem} \label{sec:open}

Originally, in this study, our construction had only one ``$C$'' vertex. This would lead to a limiting value of $\sqrt{\frac{4}{3}} = \sqrt{1+\left(\frac{1}{2}\right)^2 + \left(\frac{1}{2^2}\right)^2 + \ldots }  \approx 1.15470$ which is just a tad less than $\sqrt{\frac{67}{50}} \approx 1.15758$. 

This would seem to suggest that adding more ``$C$'' vertices (which we call the set C) would be beneficial as in Figure \ref{fig:generic}. This would result in a larger total probability mass among the $C_i$. However, splitting the probability mass among more vertices would result in a smaller 2-norm when squaring. For instance, if $\pi_{\alpha_1}$ was concentrated at one vertex whereas $\pi_{\alpha_2}$ was concentrated on more than two other vertices equally, the resulting difference would be at most 
$\sqrt{1 + 3 \frac{1}{3^2}} = \sqrt{\frac{4}{3}}$, which is worse than our result. 

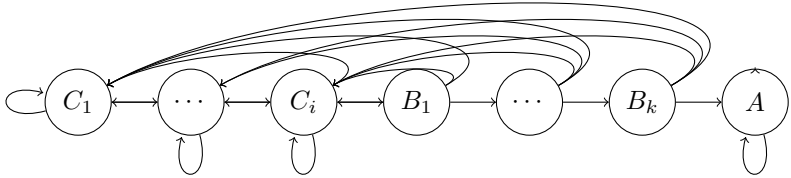
\begin{figure}
    \centering

    \begin{tikzpicture}
    \node[state] (a) at (10,4) {$A$};
    \node[state] (b) at (8.5,4) {$B_k$};
    \node[state] (c) at (7,4) {$\cdots$};
    \node[state] (d) at (5.5,4) {$B_1$};
    \node[state] (h) at (4,4) {$C_i$};
    \node[state] (e) at (2.5,4) {$\cdots$};
    \node[state] (f) at (1,4) {$C_1$};


    \path[<-] (a) edge (a);
    \path[->] (b) edge (a);
    \draw[<-] (b) edge (c);
    \draw[<-] (c) edge (d);
    \draw  (e) edge [<-, out=30, in=30] (b);
    \draw  (f) edge [<-, out=30, in=30] (c);
    \draw  (e) edge [<-, out=30, in=30] (c);
    \draw  (f) edge [<-, out=30, in=30] (h);
    \draw  (h) edge [<-, out=30, in=30] (b);
    \draw  (h) edge [<-, out=30, in=30] (c);
    \draw  (h) edge [<-, out=30, in=30] (d);
    
    \draw[->] (e) edge (f);
    \draw (f) edge [<-, out=30, in=30]  (b);
    \draw (f) edge [<-, in=30, out=30 ]  (d);
    \draw[<-] (d) edge (h);
    \draw[<-] (e) edge (h);
    \draw[->] (e) edge (h);
    \draw[<-] (e) edge (f);
    \path (a) edge [loop below] (a);
    \path (e) edge [loop below] (e);
    \path (h) edge [loop below] (h);
    \path (f) edge [loop left] (f);
    \draw[->] (d) edge  (h);
\end{tikzpicture}

    \caption{A more general construction with more $C$ vertices. The optimal value of $\| \pi_{\alpha_1} - \pi_{\alpha_2} \| $ occurs for two $C$ vertices.  }
    \label{fig:generic}
\end{figure}

Consider a similar construction to the path we had earlier but with a varying number of vertices in $C$. This general construction is the same as the previous version except we can add a fixed amount of vertices that are copies (in that they have the same in-arcs and out-arcs) of $C_1$ except they do not have an arc to $B_1$ and they have an arc to and from every other vertex in $C$.

We will give an informal argument why having two vertices in $C$ is the best for this more general construction. As in the proofs above, for $\alpha$ close to (but not equal to) $1$, the probability mass always sums to $1$ and approximately forms a geometric series, so we have
\[ 1 \approx (m -1)  \pi_{C_i} + \frac{\pi_{C_i}}{1- \frac{1}{m + 1}}\]
where $\pi_{C_i}$ is the Pagerank each vertex in any such vertex in $C$  and $m$ is the number of such vertices. By solving for $\pi_{C_i}$ we have \[\pi_{C_i} \approx \frac{m}{1 + m^2}.\] Based on this result, we can make a formula for $\| \pi \|_{2} ^{2}$
\begin{eqnarray*}
\| \pi \|_2^2  &\approx& m \pi_{C_i}^2 + \sum_{i=1}^{k} \pi_{B_i}^2    \\
 &\approx& m \pi_{C_i}^2 + \sum_{i=1}^{k} \frac{1}{(m+1)^{2i}} \pi_{C_i}^2  \\
&\approx&  \pi_{C_i}^2  \left[ m-1 + \frac{1 }{1-\frac{1}{(1+m)^2}} \right ]  \\
&=&  \frac{m^2}{\left(m^2+1\right)^2} \left[ m-1 + \frac{1 }{1-\frac{1}{(1+m)^2}} \right ] \\ 
&=& \frac{m^4+2 m^3+m}{(m+2) \left(m^2+1\right)^2}
\end{eqnarray*}

As a result, we have that 

\[ \|\pi_{\alpha_1} - \pi_{\alpha_2} \| \approx \sqrt{1+ \frac{m^4+2 m^3+m}{(m+2) \left(m^2+1\right)^2 }} \]
This function reaches a maximum of approximately $\sqrt{1.360390}$ at approximately $m=1.445036$, but since this problem only accepts discrete answers the optimum is either at $m=1$ or $m=2$. For $m=1$ we have $\sqrt{\frac{4}{3}}$, and for $m=2$, we have $\sqrt{\frac{67}{50}}$ slightly higher as shown in Figure \ref{fig:fgraph}. 

\begin{figure}[ht]
    \centering
    \includegraphics[width=4in]{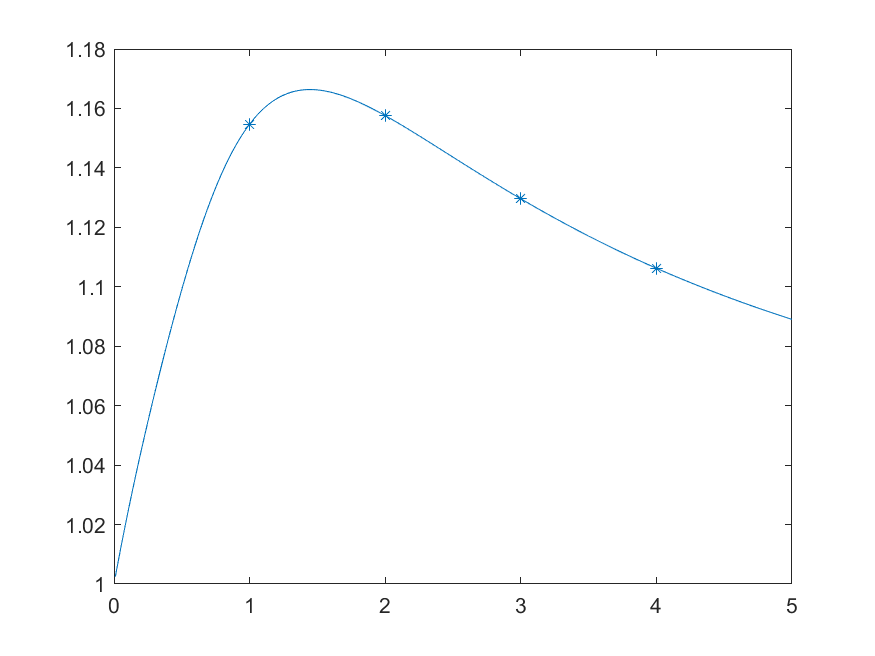}
    \label{fig:fgraph}
    \caption{
    Plot of $f(x)= \sqrt{1+ \frac{x^4+2 x^3+x}{(x+2) \left(x^2+1\right)^2}}$ with the values for $x=1, 2, \ldots$ indicated. For positive integral values $f(x)$ is maximized at  $x=2$ indicating that 2 ``C'' vertices provides for an optimal construction.
    }
\end{figure}

We can also ask about the supremum values of $\| \pi_{\alpha_1} - \pi_{\alpha_2} \|_p $ for different $p \ge 1$
Our construction 
achieves the supremum values for both the $1$- and the $\infty$-norms, 
The maximum that can be possibly achieved for the $\infty$ norm for the difference between two Pagerank vectors is 1, and these constructions achieve that because when $\alpha = 1$, $\pi_A = 1$ and when $\alpha = 1-1/k$, $\pi_A \to 0$. Similarly, the maximum that can be possibly achieved for the $1$-norm for the difference between two Pagerank vectors is 2, and our constructions achieve that because when $\alpha = 1$, $\pi_A = 1$ and when $\alpha = 1-1/k$, $\pi_A \to 0$ and so the probability mass is distributed among the other vertices and so $\| \pi_{\alpha_1} - \pi_{\alpha_2} \|_1 \to 1 $. However, it is remains unclear if our construction is optimal for $1 < p < \infty$.



We would hope the argument above is convincing that $\sqrt{\frac{67}{50}}$ is perhaps the best possible for the 2-norm, but this is far from a proof! In fact, it highlights why optimization is difficult for discrete combinatorial objects. There is much work on optimizing a graph based on eigenvalues and/or eigenvectors \cite{aksoy2016extreme, tait2019colin, tait2017three}, and even among them, there are several cases where the optimal graph follows a canonical construction in small cases but becomes different in larger cases \cite{kostochka2008ks}. 


\bibliographystyle{siamplain}
\bibliography{pagerank}

\end{document}